\documentclass[12pt]{amsart}
\usepackage[utf8]{inputenc}
\usepackage[T1]{fontenc}
\usepackage{amsfonts}
\usepackage{amsmath,amscd}
\usepackage{amsthm}
\usepackage{latexsym}
\usepackage{amssymb}
\usepackage{enumerate}
\usepackage{url}
\markleft{\hfill G. Raposo\hfill }
\usepackage{hyperref}

\newtheorem{theorem}{Theorem}[section]
\newtheorem*{theorem*}{Theorem}

\newtheorem{definition}[theorem]{Definition}
\newtheorem{example}[theorem]{Example}
\newtheorem{lemma}[theorem]{Lemma}
\newtheorem{remark}[theorem]{Remark}
\newtheorem*{claim*}{Claim}
\usepackage{tikz-cd}

\usepackage{mathdots}

\usepackage{geometry}
\usepackage{tikz}
\usetikzlibrary{arrows,matrix,positioning}

\usepackage{caption}
\usepackage{subcaption}
\usepackage{graphicx}
\usepackage{tikz}
\usetikzlibrary{decorations.pathmorphing} 
\usetikzlibrary{decorations.pathreplacing} 

\numberwithin{equation}{section}

\newcommand{\R}{\mathbb{R}} 
\newcommand{\C}{\mathbb{C}} 
\newcommand{\N}{\mathbb{N}} 
\newcommand{\E}{\mathbb{E}} 

\allowdisplaybreaks

\title{Interpolation of Gaussian Free Fields via Random Matrices}
\author{Gabriel Raposo}
\address{Department of Mathematics, University of Toronto, \newline 40 St. George Street, Toronto, ON M5S2E4, Canada}
\email{gabo.raposo@utoronto.ca}

\begin{document}

\maketitle

\begin{abstract} In this paper we construct a two parameter family of generalized Gaussian fields which interpolates between the Gaussian Free Field, conditioned Gaussian Free Fields, and Gaussian Free Fields with independent noise. Depending on the values of the parameters, we characterize these generalized Gaussian fields as partial conditionings of a Gaussian Free Field with perturbations of its first two Fourier modes. We show that this family arises as the fluctuations of the height function associated with the corner process of a two parameter variant of Wigner random matrices.
\end{abstract}

\tableofcontents

\section{Introduction}
\subsection{Overview}
The Gaussian Free Field has been a central object in the asymptotic study of various statistical mechanics models. Classical examples include dimer models, interacting particle systems, and six vertex models, among many others \cite{Ke09,BF,DKKMO20,Gorin,GN25,DKLM26}. In recent years, various relatives of the Gaussian Free Field have appeared as scaling limits, such as correlated Gaussian Free Fields \cite{Bo,BB14}, Gaussian Free Field with Brownian motion \cite{BPZ25}, Gaussian Free Field with a discrete component \cite{BN25} and conditioned Gaussian Free Fields \cite{RaMain}. 

The goal of this paper is to construct via random matrix methods a two parameter family of Gaussian fields interpolating between the Gaussian Free Field, the conditioned GFF of \cite{RaMain} and the GFF with noise of \cite{BPZ25}, while also producing a new higher-conditioned GFF. The two parameters arise from the second moment of the diagonal entries and the fourth moment of the off-diagonal entries of a Wigner matrix. In the limit these correspond to the first two Fourier modes of the resulting Gaussian field. This suggests that random matrix models provide a mechanism to construct new relatives of the Gaussian Free Field.

Note that the appearance of the Gaussian Free Field in the description of global fluctuations of the height function associated to the corner process of random matrices has been extensively studied during the last decade, see \cite{Bo,BoG,DP18,LRS20,KZ23}. To ease the notation, we present the definitions corresponding to complex Hermitian ($\beta=2$) random matrices. However, the definitions can be immediately extended to real symmetric ($\beta=1$) and quaternionic self-adjoint ($\beta=4$) random matrices (Corresponding to unitary, orthogonal, and symplectic ensembles, respectively).

\begin{definition}\label{DefinitionModel}
Let $n\in \N$ and let $\big(X_{i,j}\big)_{1\leq i<j\leq n}$ be a collection of i.i.d. complex ($\beta=2$) random variables and $\big(X_{\ell,\ell}\big)_{1\leq \ell\leq n}$ be an independent i.i.d collection of real random variables such that $\E|X_{i,j}|^k<\infty$, $\E\big[X_{i,j}\big]=0$  for any $1\leq i\leq j \leq n$ and $k\geq 1$, $\E\big[X_{1,2}^2\big]=0$ and $\E\big[|X_{1,2}|^2\big]=1$. We further fix $\E\big[X_{1,1}^2\big]=a$ and $\E\big[|X_{1,2}|^4\big]=b$. Finally, for $i<j$ fix $X_{j,i}:=\bar{X}_{i,j}$. We say that $M$ is an $(a,b)$-Wigner random matrix if
\begin{equation*}M:=\begin{bmatrix}
    X_{1,1} & X_{1,2} & X_{1,3}  & \cdots & X_{1,n}\\
    X_{2,1} & X_{2,2} & X_{2,3} & \cdots & X_{2,n}\\
    X_{3,1} & X_{3,2} & X_{3,3} & \cdots & X_{3,n}\\
    \vdots & \vdots & \vdots & \ddots  \\
    X_{n,1} & X_{n,2}& X_{n,3} & \cdots & X_{n,n} 
\end{bmatrix}.
\end{equation*}
\end{definition}

\begin{example}\label{ExampleOfDefinitions}
One can choose $M$ to be a GUE random matrix, in which case we have $a=1$ and $b=2$. Another choice is to have $X_{i,j}$ to be uniformly distributed on the unit circle when $i\neq j$ and $0$ when $i=j$, in which case we have $a=0$ and $b=1$. In the real symmetric case, this corresponds to GOE, and to take $X_{i,j}$ to be uniformly distributed on $\{-1,1\}$ when $i\neq j$ and $0$ when $i=j$, respectively.
\end{example}

\begin{figure}[h]
    \centering
    \subfloat[\centering Random height function.]{{\includegraphics[scale=0.45]{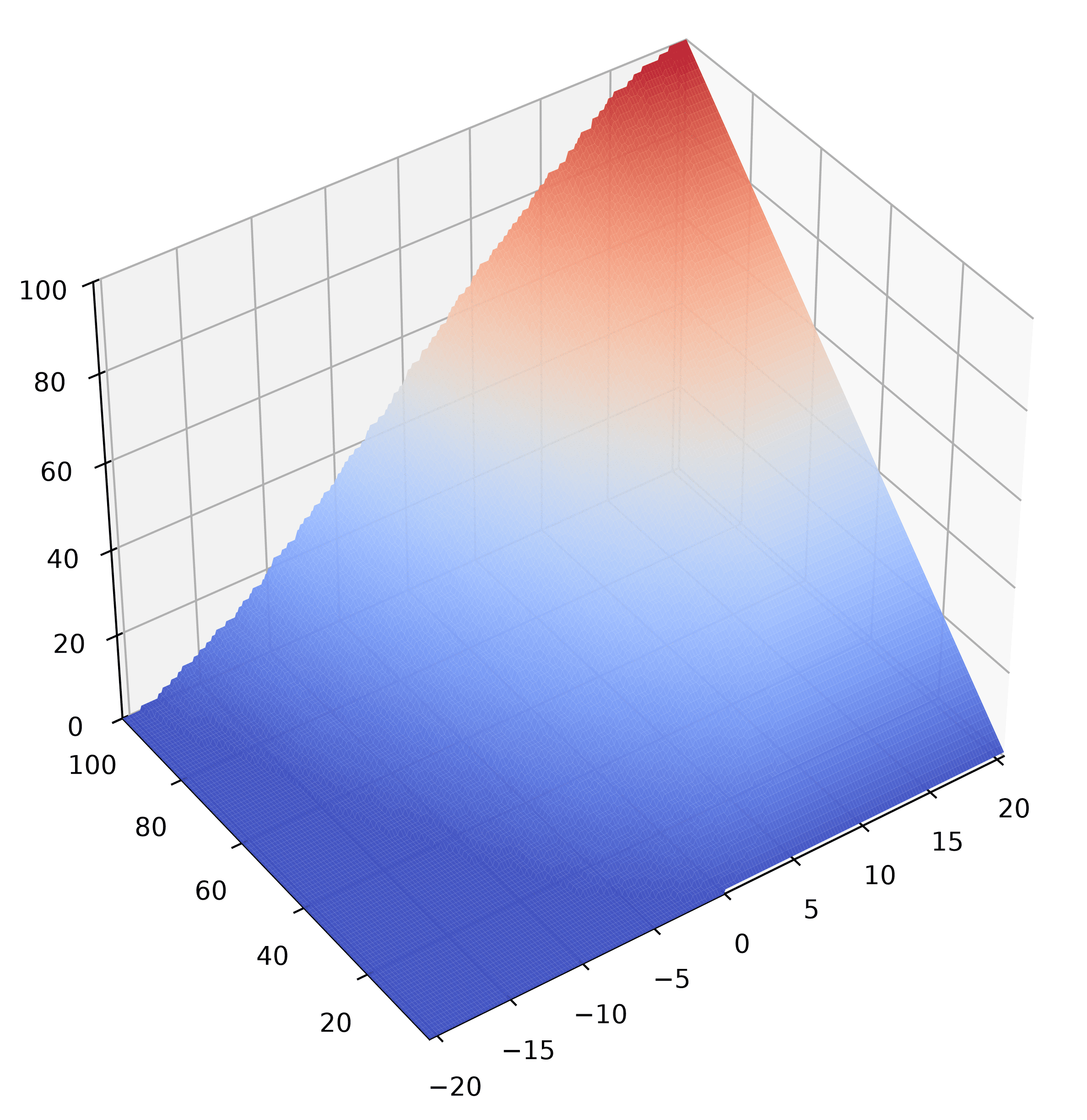} }}%
    \qquad
    \subfloat[\centering Fluctuation surface.]{{\includegraphics[scale=0.45]{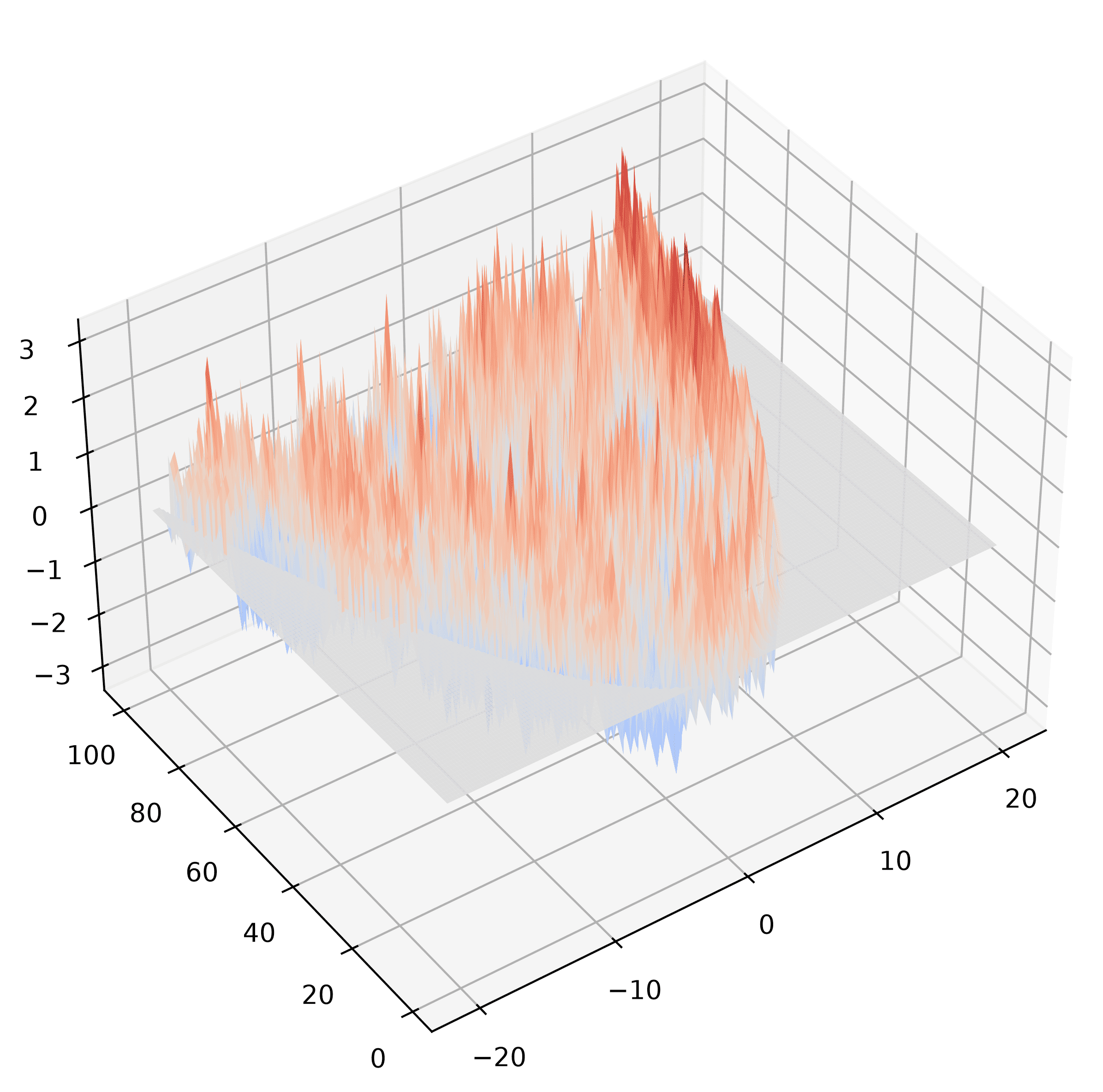} }}%
    \caption{Height function $\textup{H}$ and fluctuations $\sqrt{\pi}\big(\textup{H}-\mathbb{E}\textup{H}\big)$ for a $100\times 100$ random matrix with $0$ diagonal and uniform on the unit circle random entries.}%
\label{Figure1}%
\end{figure}

Note that for any $(a,b)$-Wigner random matrix its empirical measure converges to the semicircle distribution, this is a standard fact that can be found in \cite[Chapter 1]{AGZ}.

\begin{lemma}\label{TheoSemiCircleDist}
Let $M$ be an $(a,b)$-Wigner random matrix with eigenvalues $\lambda_1,\dots,\lambda_n$. Then its empirical spectral distribution $\frac{1}{n}\sum_{i=1}^n \delta_{\lambda_i/\sqrt{n}}$ weakly converges in probability to the semicircle distribution $\frac{1}{2\pi}\sqrt{4-x^2} \,\mathbf{1}_{[-2,2]}dx$.
\end{lemma}

We now introduce the language needed to state our main results. For each $n\times n$ random matrix $M$ and any integer $n'<n$, denote $M_{n'}$ the $n'\times n'$ upper left corner of $M$. We are interested in the height function $\textup{H}$ associated with a random matrix $M$ defined by 
\begin{equation*}
    \textup{H}(x,y):=\#\{\textup{Eigenvalues of }M_{\lfloor y \rfloor} \textup{ that are less than or equal to } x\}.
\end{equation*}
See Figure \ref{Figure1} for a plot of a random surface for a $(0,1)$-Wigner random matrix. 

The fluctuations of this height function are described by a relative of the Gaussian Free Field. See  \cite{Sh07,PW} for an introduction to the Gaussian Free Field. Here we consider the generalized Gaussian field $\mathfrak{C}_{a,b}$ on the upper-half plane $\mathbb{H}:=\{z\in \C:\Im(z)>0\}$ with covariance kernel
\begin{equation*}
\frac{-1}{2\pi}\ln\Biggr| \frac{z-w}{z-\bar{w}}\Biggr|
-\frac{(1-a)\min\big(|z|^2,|w|^2\big)}{\pi}\Im\Big(\frac{1}{z}\Big)\Im\Big(\frac{1}{w}\Big)-\frac{(2-b)\min\big(|z|^4,|w|^4\big)}{2\pi}\Im\Big(\frac{1}{z^2}\Big)\Im\Big(\frac{1}{w^2}\Big).
\end{equation*}

Interestingly, the covariance kernel depends on the pair $(a,b)$. See Figure \ref{Figure2} for a sample of the limiting Gaussian field $\mathfrak{C}_{0,1}$ in the upper-half plane for an $800\times 800$ $(0,1)$-Wigner random matrix. See Figure \ref{DiagramOfBehaviors} for a diagram of the different Gaussian Free Field relatives the Gaussian field $\mathfrak{C}_{a,b}$ corresponds to, Table \ref{TableOfBehaviors} for a summary and Theorem \ref{Theorem2} for a precise description. 

\begin{figure}[h]
    \centering{{\includegraphics[scale=0.66]{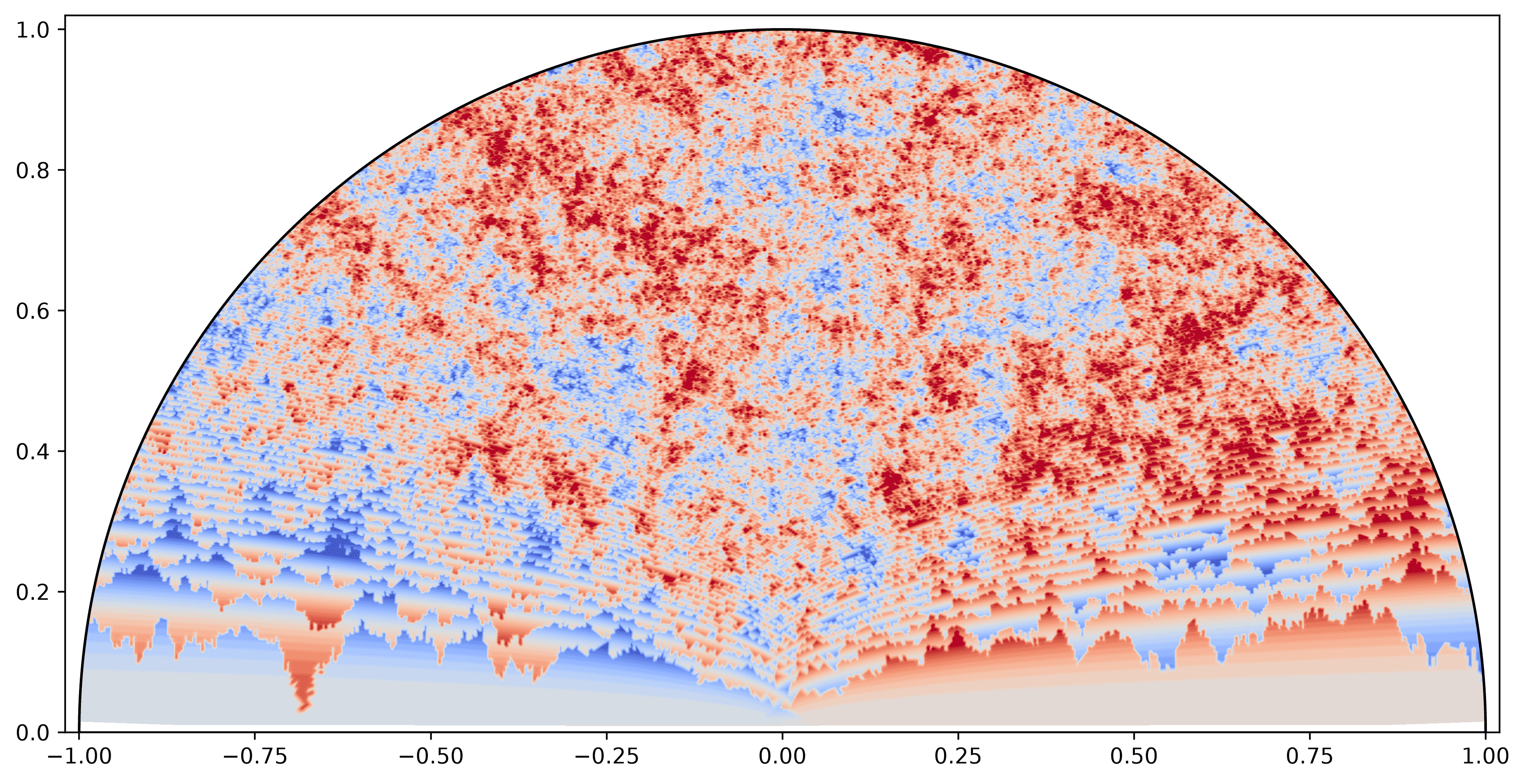}}}%
    \caption{Fluctuations for a $(0,1)$-Wigner random matrix of size $800\times 800$ in the upper-half plane after rescaling.}%
\label{Figure2}%
\end{figure}

We identify the interior of the parabola with the upper-half plane $\mathbb{H}$ by means of the map $\Omega(x,t)=\frac{x}{2}+i\sqrt{t-\frac{x^2}{4}}$. Then $\mathfrak{C}_{a,b}(x,t)$ denotes the pullback of the field $\mathfrak{C}_{a,b}$ on $\mathbb{H}$ under $\Omega$.

\begin{theorem}\label{TheoremModel}
Let $M$ be as in Definition \ref{DefinitionModel} and $\textup{H}$ be the associated height function, then
    \begin{equation*}
    \sqrt{\pi}\bigg[\textup{H}(\sqrt{n}x,nt)-\E\textup{H}(\sqrt{n}x,nt)\bigg]\to \mathfrak{C}_{a,b}(x,t), \hspace{1mm} \text{ as }  \hspace{1mm} n\to \infty.
    \end{equation*}
That is, as $n \to \infty$, the collection of random variables 
\begin{equation*}
 \sqrt{\pi}\int_{-\infty}^{+\infty} x^k \bigg[\textup{H}(\sqrt{n}x,nt)-\E\textup{H}(\sqrt{n}x,nt)\bigg]\,dx
\end{equation*}
converges, in the sense of moments, to 
\begin{equation*}
\int\limits_{\substack{z\in \mathbb{H},\, |z|^2=t}} x(z)^k \mathfrak{C}_{a,b}(z)\frac{dx(z)}{dz}\,dz, \hspace{5mm}\textup{where }\hspace{1mm} x(z)=2\Re(z).
\end{equation*}
\end{theorem}

Note that when $a=1$ and $b=2$ we recover the standard Gaussian Free Field. If $a=0$ and $b=2$ we recover the conditioned GFF from \cite{RaMain}, if $a=2$ and $b=2$ we recover a GFF perturbed by a Brownian motion \cite{BPZ25} and if $b=1$ with $a=0$ or $a=1$ we obtain two instances of a new higher conditioned GFF (CGFF). It is only natural to ask if there is an appropriate description for the generalized Gaussian field $\mathfrak{C}_{a,b}$, the answer is in the affirmative. Indeed, $\mathfrak{C}_{a,b}$ corresponds to a Gaussian Free Field with a perturbation to the lower Fourier modes.  

\begin{figure}[h]
\centering
\vspace{-3mm}
\begin{tikzpicture}[scale=1.5]

\draw[->] (0,0) -- (6,0) node[right] {$a$};
\draw[->] (0,0) -- (0,6) node[above] {$b$};

\draw[dashed] (3,0) -- (3,6);
\draw[dashed] (0,3) -- (6,3);

\draw (0,0.08) -- (0,-0.08) node[below, , yshift=-1mm] {$0$};
\draw (3,0.08) -- (3,-0.08) node[below, yshift=-1mm] {$1$};
\draw (0.08,0) -- (-0.08,0) node[left, xshift=-1mm] {$1$};
\draw (0.08,3) -- (-0.08,3) node[left, xshift=-1mm] {$2$};

\fill (3,3) circle (2pt);
\node[above right] at (3,3) {GFF};

\fill (0,0) circle (2pt);
\node[above right] at (0,0) {CGFF};

\fill (3,0) circle (2pt);
\node[above right] at (3,0) {CGFF};

\fill (0,3) circle (2pt);
\node[above right] at (0,3) {CGFF};

\node[align=center] at (4.5,4.5)
{GFF with\\Gaussian noise};

\node[align=center] at (1.5,1.5)
{Partially\\Conditioned GFF};

\node[align=center] at (1.5,4.5)
{Partially\\ Conditioned GFF with\\Gaussian noise};

\node[align=center] at (4.5,1.5)
{Partially \\ Conditioned GFF with\\Gaussian noise};

\end{tikzpicture}
\caption{Diagram of behaviors for the Gaussian field $\mathfrak{C}_{a,b}$.}\label{DiagramOfBehaviors}
\end{figure}
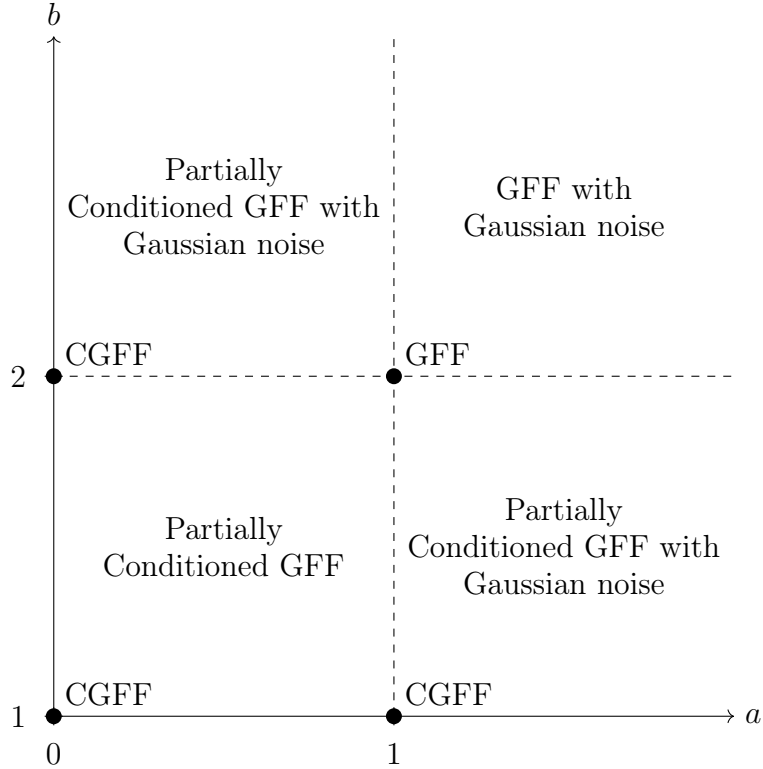

Informally, for a $(a,b)$-Wigner random matrix, one can write down a formal Fourier expansion for the induced generalized Gaussian field. Denoting $t=|z|^2$ and $z=\sqrt{t}e^{i\theta}$ for $t> 0$ and $0<\theta<\pi$ we have the following
\begin{equation*}
\mathfrak{C}_{a,b}(z)=-\sqrt{\frac{a}{\pi t }} B_1\big(t\big) \sin(\theta)-\sqrt{\frac{b-1}{2\pi t^2}} B_2\big(t^2\big)\sin(2\theta)-\sum_{j=3}^\infty \frac{1}{\sqrt{\pi j t^j}} B_j\big(t^j\big)\sin(j\theta).
\end{equation*}
Here $\big(B_j(\cdot)\big)_{j=1}^\infty$ is a collection of independent Brownian motions. The standard Gaussian Free Field corresponds to the case $a=1$ and $b=2$. Hence, in our case we obtained a formal Fourier expansion of a Gaussian Free Field with first Fourier mode rescaled by $\sqrt{a}$ and second Fourier mode rescaled by $\sqrt{b-1}$. 

Using this heuristic we can recover a more precise mathematical description. Indeed, for $a>1$ or $b>2$ the resulting covariance can be recovered by adding independent noise to the first mode or the second mode of the GFF. More generally, note that for any Gaussian random variable $X$ with variance $\sigma^2$ and independent Gaussian noise $\eta$ with variance $p\sigma^2$ fixing $Y=X+\eta$ and conditioning on $Y=0$, we get $\textup{Var}\big(X|Y=0\big)=\frac{p}{p+1}\sigma^2$. Hence, for $a<1$ or $b<2$, the Gaussian field covariance can be recovered by adding independent noise to each mode and then conditioning on the mode to be $0$.

\begin{definition}
We will say that a Gaussian field $\mathfrak{C}$ is conditioned to have first mode $0$, if it is conditioned so that $\int_{0}^\pi \mathfrak{C}(\sqrt{t}e^{i\theta}) \sin(\theta) d\theta=0$ for any $t>0$. Similarly, we will say that we are conditioning to have second mode $0$ if it is conditioned so that $\int_{0}^\pi \mathfrak{C}(\sqrt{t}e^{i\theta}) \sin(2\theta) d\theta=0$ for any $t>0$.
\end{definition}

Putting everything together, there are two critical values. For the first mode, one of the critical values $a=0$ induces a pure conditioning to the GFF, when $0<a<1$ we have a partial conditioning, that is we condition a GFF after adding an independent Gaussian noise. For the second critical value $a=1$ we recover the GFF, and finally for $a>1$ we have a GFF with an independent Gaussian noise. The same happens with the second mode, where the critical values are $b=1$ and $b=2$. See Figure \ref{DiagramOfBehaviors} and Table \ref{TableOfBehaviors}. 

\begin{theorem}\label{Theorem2}
Let $\mathfrak{C}_{a,b}$ be the generalized Gaussian field obtained from Theorem \ref{TheoremModel}. Fix a standard Gaussian Free field $\mathfrak{G}$ in the upper-half plane, $\eta_1(z)=\tfrac{-1}{\sqrt{\pi t}} B_1(t) \sin(\theta)$ and $\eta_2(z)=\tfrac{-1}{\sqrt{2 \pi t^2}} B_2(t^2) \sin(2\theta)$ for $B_1$ and $B_2$ two independent Brownian motions that are also independent of $\mathfrak{G}$. Then $\mathfrak{C}_{a,b}$ can be characterized as follows.
\begin{enumerate}
    \item If $(a,b)=(1,2)$, $\mathfrak{C}_{a,b}$ corresponds to a standard Gaussian Free Field $\mathfrak{G}$.
    \item If $(a,b)=(0,1)$, $\mathfrak{C}_{a,b}$ corresponds to $\mathfrak{G}$ conditioned to have first and second modes $0$. If $(a,b)=(0,2)$, $\mathfrak{C}_{a,b}$ corresponds to $\mathfrak{G}$ conditioned to have first mode $0$. If $(a,b)=(1,1)$, $\mathfrak{C}_{a,b}$ corresponds to $\mathfrak{G}$ conditioned to have second mode $0$.
    \item If $a\geq1$ and $b\geq 2$, $\mathfrak{C}_{a,b}$ corresponds to $\mathfrak{G}+\sqrt{a-1}\eta_1+\sqrt{b-2}\eta_2$.
    \item If $0\leq a <1$ and $1\leq b <2$, $\mathfrak{C}_{a,b}$ corresponds to $\mathfrak{G}$ conditioned so that $\mathfrak{G}+\sqrt{\frac{a}{1-a}}\eta_1+\sqrt{\frac{b-1}{2-b}}\eta_2$ has first and second modes $0$. 
    \item  If $0\leq a <1$ and $b\geq 2$, $\mathfrak{C}_{a,b}$ corresponds to $\mathfrak{G}+\sqrt{b-2}\eta_2$ conditioned so that $\mathfrak{G}+\sqrt{\frac{a}{1-a}}\eta_1$ has first mode $0$. If $a\geq1$ and $1\leq b<2$, $\mathfrak{C}_{a,b}$ corresponds to $\mathfrak{G}+\sqrt{a-1}\eta_1$ conditioned so that $\mathfrak{G}+\sqrt{\frac{b-1}{2-b}}\eta_2$ has second mode $0$. 
\end{enumerate}
\end{theorem}

Note that the case $(a,b)=(0,2)$ of Theorem \ref{Theorem2} is already present in \cite[Proposition 3.5]{RaMain}. Additionally, case $(a,b)=(2,2)$ of Theorem \ref{Theorem2} is present in \cite[Section 4.2]{BPZ25}. Interestingly, both Gaussian Free Field relatives arise naturally from combinatorial models.

\begin{remark}
Note that the critical values are $\beta$-dependent. Indeed, if rather than Hermitian random matrices ($\beta=2$), we consider real symmetric ($\beta=1$) or quaternionic self-adjoint ($\beta=4$), then the critical values for the first mode are $a=0$ and $a=2/\beta$ while the critical values for the second mode are $b=1$ and $b=1+2/\beta$.
\end{remark}

The proof of Theorem \ref{TheoremModel} is based on the moment method and adapts the ideas outlined in \cite{AGZ,Bo}. We review this set of tools at the beginning of Section \ref{section2}. The proof of Theorem \ref{Theorem2} is based on a diagonalization of the Gaussian Free Field covariance in terms of Chebyshev polynomials of the first kind, equivalently its Fourier expansion. We apply this technique to each one of the cases of the theorem to conclude. 

\begin{table}[h]
    \centering
\begin{tabular}{ |c|c| } 
 \hline
 $(a,b)$ & GFF relative  \\ 
 \hline 
 $(1,2)$ & Standard GFF  \\ 
 \hline
 $(0,1)$, $(0,2)$, or $(1,1)$ & Conditioned GFF  \\ 
 \hline
  $a\geq1$, $b\geq2$ and $(a,b)\neq(1,2)$ & GFF with independent Gaussian noise\\
 \hline
  $a=0$ and $b>2$, or & Conditioned GFF with independent \\ 
 $a>1$ and $b=1$ & Gaussian noise\\
 \hline
  $0\leq a \leq 1$ and $1\leq b \leq 2$ & Partially conditioned GFF  \\
 \hline
 $0\leq a<1$ and $b>2$, or & Partially conditioned GFF with \\ 
 $a>1$ and $1\leq b<2$ & independent Gaussian noise.\\
 \hline
\end{tabular}
    \caption{Summary of behaviors for the Gaussian field $\mathfrak{C}_{a,b}$.}
    \label{TableOfBehaviors}
\end{table}

\subsection{Acknowledgments}
I am grateful to my advisor Vadim Gorin for the very valuable discussions and corrections throughout all this work. The project was partially supported by NSF grant DMS -- 2246449.

\section{Proof of Theorem \ref{TheoremModel}}\label{section2}
We start by briefly describing the moment method in the random matrix theory setting which we use to prove the first part of Theorem \ref{TheoremModel}. We closely follow \cite{Bo}, where this technique is used to identify Gaussian Free Field fluctuations in nested corners of Wigner random matrices. In turn, \cite{Bo} closely follows \cite[Section 2.1]{AGZ} where a detailed account of the moment method is present.

\begin{lemma} \label{LemmaMomentMethod}
Let $r\geq 1$, $k_1,\dots,k_r$ and $n_1,\dots,n_r$ be positive integers. For each $1\leq i\leq r$ denote by $m_{n_i}$ the empirical measure of the random matrix $M_{n_i}$ and $\xi_{k_i,n_i}=\int_\R x^{k_i} \, dm_{n_i}(x)$ its (random) moments, then
\begin{equation*}
\E\bigg[\prod_{i=1}^r \xi_{k_i,n_i}\bigg]=\E\bigg[\prod_{i=1}^r\textup{tr}\Big(M_{n_i}^{k_i}\Big)\bigg].
\end{equation*}
\end{lemma}
\begin{proof}
This follows immediately from the identity $\xi_{k_i,n_i}=\textup{tr}\Big(M_{n_i}^{k_i}\Big)$ for each $i=1,\dots,r$.
\end{proof}

A beautiful consequence of the lemma above is that any computation of the joint moment of spectral measures is reduced to a combinatorial problem.  

\begin{figure}[h!] 
    \centering
\includegraphics[scale=0.4,angle=90]{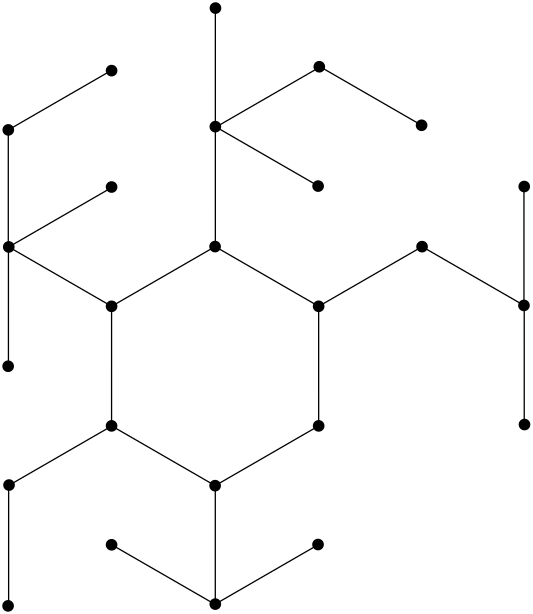}\\
\caption{Unicycle graph with 25 vertices and a cycle of length $6$.}\label{FigureUnicyclegraph}
\end{figure}

More precisely, the computation of the joint moment in Lemma \ref{LemmaMomentMethod} will be reduced to the count of a certain collection of $r$ cyclic graphs each with $k_i$ edges on the complete graphs with vertices on the set $\{1,2,\dots,n_i\}$ for $i=1,\dots,r$. In practice, the computation of the moments of $\xi_{k,n}$ is reduced to the count of trees, the computation of the covariance is reduced to the count of pairs of unicycle graphs, that is graphs containing a single cycle, or equivalently a cycle with trees attached to its vertices, see for example Figure \ref{FigureUnicyclegraph},  and the computation leading to Gaussianity consists of noticing a necessity of count of perfect matching of pairs of unicyclic graphs for the higher joint moments. 

Similarly to the proof of the semicircle law, the proof of the Gaussianity of the joint process $\big(\xi_{k_i,n_i}\big)_{i=1}^r$ follows the arguments already present in \cite{AGZ,Bo}. For this reason we focus on the computation of the covariance. Start by noticing that the covariance formula for the single-level setting, that is for $r=2$, $n=n_1=n_2$, and denoting $\xi_{k,n}=\xi_{k_1,n_1}$, $\xi_{k',n}=\xi_{k_2,n_2}$ gives 
\begin{equation*}
\begin{split}
\lim_{n \to \infty} \frac{1}{n^{\frac{k+k'}{2}}} \textup{Cov}(\xi_{k,n},\xi_{k',n})=\binom{k}{\frac{k-1}{2}}\binom{k'}{\frac{k'-1}{2}}\E[X_{1,1}^2]+2\binom{k}{\frac{k-2}{2}}\binom{k'}{\frac{k'-2}{2}}\big(\E[|X_{1,2}|^4]-1\big)\\+\sum_{i=3}^{\infty}i\binom{k}{\frac{k-i}{2}}\binom{k'}{\frac{k'-i}{2}}.
\end{split}
\end{equation*}

This formula is indeed obtained in \cite[Proof of Theorem 2']{Bo} from \cite[Equation (2.1.44)]{AGZ} by using the Catalan $k$-fold convolution \cite{Re12}, note that in \cite{AGZ} an extra factor of $2$ appears. This is because their proof is for real symmetric matrices rather than complex Hermitian matrices. In our setting we immediately have that $\E[|X_{1,1}|^2]=a$ and $\E\big[|X_{1,2}|^4\big]=b$, from which we get that
\begin{equation}\label{FormulaCovarianceUniCycles}
\begin{split}
\lim_{n \to \infty} \frac{1}{n^{\frac{k+k'}{2}}} \textup{Cov}(\xi_{k,n},\xi_{k',n})=(a-1)\binom{k}{\frac{k-1}{2}}\binom{k'}{\frac{k'-1}{2}}+2(b-2)\binom{k}{\frac{k-2}{2}}\binom{k'}{\frac{k'-2}{2}}\\+\sum_{i=1}^{\infty} i\binom{k}{\frac{k-i}{2}}\binom{k'}{\frac{k'-i}{2}}.
\end{split}
\end{equation}

Importantly each pair of binomial coefficients with index $i$ in the equation above corresponds to the number of pairs of unicycle graphs that coincide over a cycle with $i$ vertices. 

\begin{remark}\label{RemarkChebychevPolynomyals}
Starting from equation (\ref{FormulaCovarianceUniCycles}) above and using the inversion formulas for the Chebyshev polynomials of the first kind, it is a short computation to verify that the random variables $\big(\int_\R T_i(x) \, dm_{n}(x)\big)_{i}$, where $\big(T_i(x)\big)_i$ denotes the Chebyshev polynomials and $m_{n}$ denotes the empirical measure of the random matrix $M_n$, converges to a sequence of independent Gaussian random variables, see for example \cite[Proposition 3]{Bo}.
\end{remark}

We are interested in adapting formula (\ref{FormulaCovarianceUniCycles}) to the multilevel setting and recover the covariance formula for the height function for it. We follow a standard argument from \cite{Bo}.

\begin{proof}[Proof of Theorem \ref{TheoremModel}]
Denote $n_1=\lfloor t_1 n  \rfloor$ and $n_2=\lfloor t_2 n  \rfloor$ for $0\leq t_1,\,t_2\leq1 $ and let $\xi_{k,n_i}=\int_\R x^k \, dm_{n_i}(x)$ for $i=1,2$. We want to adapt the combinatorial formula for the covariance to the setting $n_1\neq n_2$. In this setting when determining the leading coefficient for the number of pairs of unicycle graphs having a cycle that coincide we do the following. 

The first graph has vertices on the set $\{1,2,\dots,n_1\}$ while the second graph has vertices on the set $\{1,2,\dots,n_2\}$. Because the vertices in the cycles must coincide, the vertices must take values on the set $\{1,\dots,\min(n_1,n_2)\}$. This will cause the appearance of the term 
\begin{equation*}
    \lim_{n\to \infty}\frac{\min(n_1,n_2)^i}{n^i}=\min(t_1,t_2)^i
\end{equation*}
for pairs of unicycles with a cycle of length $i$. In addition to the $i$ vertices from the cycles, the first graph will contain $\frac{k-i}{2}$ additional vertices taking values in the set $\{1,\dots,n_1\}$ and the second graph will contain $\frac{k'-i}{2}$ additional vertices taking values in the set $\{1,\dots,n_2\}$, each corresponds to the disjoint trees attached to the vertices of the cycle. These produce the terms 
\begin{equation*}
    \lim_{n\to \infty} \Big(\frac{n_1}{n}\Big)^{\tfrac{k-i}{2}}=t_1^{\tfrac{k-i}{2}}\,\hspace{2mm}\textup{ and }\hspace{2mm}\,\lim_{n\to \infty} \Big(\frac{n_2}{n}\Big)^{\tfrac{k'-i}{2}}=t_2^{\tfrac{k'-i}{2}}.
\end{equation*}

Putting all together, jointly with the leading coefficient for the number of equivalent pairs of unicycle graphs, we get that
\begin{equation}\label{Eq1Covariance}
\begin{split}
\lim_{n \to \infty} \frac{1}{n^{\frac{k+k'}{2}}} \textup{Cov}(\xi_{k,n_1},\xi_{k',n_2})=(a-1)\binom{k}{\frac{k-1}{2}}\binom{k'}{\frac{k'-1}{2}}\min(t_1,t_2) t_1^{\frac{k-1}{2}} t_2^{\frac{k'-1}{2}}+2(b-2) \\  \times\binom{k}{\frac{k-2}{2}} \binom{k'}{\frac{k'-2}{2}} \min(t_1,t_2)^2 t_1^{\frac{k-2}{2}} t_2^{\frac{k'-2}{2}}+\sum_{i=1}^{\infty} i\binom{k}{\frac{k-i}{2}}\binom{k'}{\frac{k'-i}{2}}\min(t_1,t_2)^i t_1^{\frac{k-i}{2}} t_2^{\frac{k'-i}{2}}.
\end{split}
\end{equation}
Using this formula we now only need to compute the covariance for the height function. Denoting the random variables
\begin{equation*}
\mathcal{M}_{k,t}:=\sqrt{\pi}\int_{-\infty}^{+\infty} x^k \big[\textup{H}(\sqrt{n}x,nt)-\E\textup{H}(\sqrt{n}x,nt)\big]\,dx,
\end{equation*}
a change of variables and integrating by parts gives that
\begin{equation*}
\begin{split}
   \mathcal{M}_{k,t}=\frac{-\sqrt{\pi}}{n^{(k+1)/2}(k+1)}\Big(\int_\R x^{k+1} \, dm_{\lfloor tn \rfloor}(x)-\E\int_\R x^{k+1} \, dm_{\lfloor tn \rfloor}(x)\Big).
\end{split}
\end{equation*}
Now a direct application of the covariance formula (\ref{Eq1Covariance}) gives that
\begin{equation*}
\begin{split}
    \lim_{n\to \infty} \textup{Cov}(\mathcal{M}_{k,t_1},\mathcal{M}_{k',t_2}) =\frac{\pi}{(k+1)(k'+1)} \Bigg( (a-1)\binom{k+1}{k/2} \binom{k'+1}{k'/2} \min(t_1,t_2)t_1^{k/2}t_2^{k'/2}\\+2(b-2)\binom{k+1}{\frac{k-1}{2}} \binom{k'+1}{\frac{k'-1}{2}} \min(t_1,t_2)^2t_1^{\frac{k-1}{2}}t_2^{\frac{k'-1}{2}} \\+\sum_{i=1}^{\infty} i\binom{k+1}{\frac{k+1-i}{2}}\binom{k'+1}{\frac{k'+1-i}{2}} \min(t_1,t_2)^i t_1^{\frac{k+1-i}{2}} t_2^{\frac{k'+1-i}{2}}\Bigg).
\end{split}
\end{equation*}

We are left to analyze the combinatorial formula at the right-hand side of the equation above. By using the binomial expansion we have

\begin{align*}
\lim_{n\to \infty} \textup{Cov}(\mathcal{M}_{k,t_1},\mathcal{M}_{k',t_2}) 
&=\frac{\pi}{(2\pi i)^2(k+1)(k'+1)}\oint\limits_{\substack{ |z|^2=t_1}}\oint\limits_{\substack{ |w|^2=t_2}} \big(z+\frac{t_1}{z}\big)^{k+1}\big(w+\frac{t_2}{w}\big)^{k'+1}\\
&\hspace{-20mm}\bigg(\frac{\min(t_1,t_2)}{t_1(\min(t_1,t_2)z/t_1-w)^2}+(a-1)\frac{\min(t_1,t_2)}{z^2w^2}+2(b-2)\frac{\min(t_1,t_2)^2}{z^3w^3}\bigg)\,dz\,dw\\
&=\frac{\pi}{(2\pi i)^2}\oint\limits_{\substack{ |z|^2=t_1}}\oint\limits_{\substack{ |w|^2=t_2}} \big(z+\frac{t_1}{z}\big)^k\big(w+\frac{t_2}{w}\big)^{k'} \bigg(\ln\Big(\frac{\min(t_1,t_2)}{t_1}z-w\Big)\\
&\hspace{5mm}+(a-1)\frac{\min(t_1,t_2)}{zw}+\frac{(b-2)\min(t_1,t_2)^2}{2z^2w^2}\bigg)\frac{dx(z)}{dz}\frac{dx(w)}{dw}\,dz\,dw\\
&=\oint\limits_{\substack{z\in \mathbb{H},\\\, |z|^2=t_1}} \oint\limits_{\substack{w\in \mathbb{H},\\\, |w|^2=t_2}} \hspace{-1mm} x(z)^k x(w)^{k'}\bigg(\frac{-1}{2\pi}\ln\Biggr| \frac{z-w}{z-\bar{w}}\Biggr|-\frac{(1-a)\min\big(|z|^2,|w|^2\big)}{\pi}\\
&\hspace{-20mm}\times\Im\Big(\frac{1}{z}\Big)\Im\Big(\frac{1}{w}\Big)-\frac{(2-b)\min\big(|z|^4,|w|^4\big)}{2\pi}\Im\Big(\frac{1}{z^2}\Big)\Im\Big(\frac{1}{w^2}\Big)\bigg) \frac{dx(z)}{dz}\frac{dx(w)}{dw}\,dz\,dw.
\end{align*}

Here we first integrated by parts and then simplified the remaining terms by partitioning the integral domain with respect to $\{z: |z|^2=t_1, \Im(z)>0\}$ and $\{w: |w|^2=t_2, \Im(w)>0\}$, and its conjugates. The last identity concludes the proof of the theorem.
\end{proof}

\section{Proof of Theorem \ref{Theorem2}}

We briefly review the meaning of conditioning in the setting of Gaussian processes as discussed in \cite{LG}. Denote by $\mathcal{H}$ the Hilbert space formed by finite linear combinations of instances of a generalized Gaussian field $\mathfrak{G}$ integrated with respect to real valued polynomials in $x(z)=z+\bar{z}=2\Re(z)$ over the curves $\mathcal{C}_t:=\{z\in \mathbb{H}:|z|^2=t\}$ with inner product given by its covariance and $K$ a closed linear subspace of $\mathcal{H}$. We say that $\mathfrak{C}$ is the conditioned generalized Gaussian field $\mathfrak{G}$ with respect to $K$ if there exists an orthogonal projection $\textup{P}:\mathcal{H}\to K$ such that $\mathfrak{C}=\mathfrak{G}-\textup{P}[\mathfrak{G}]$. It is enough to check that their covariances coincide, see \cite[Corollary 1.10]{LG} for a precise statement on how we can interpret this as a conditioning of a Gaussian process. To facilitate the reading of the proof of Theorem \ref{Theorem2}, we separate the proof in parts. We start with a proof of item $(2)$.

\begin{proof}[Proof of Theorem \ref{Theorem2}, item (2)]
We specialize the proof to the case $(a,b)=(0,1)$, the other two cases being completely analogous. Denote $\mathfrak{G}$ the Gaussian Free Field on the upper-half plane as described above and $\mathfrak{C}_{0,1}$ the one from the introduction of the paper. In our setting, we are integrating the Gaussian field over curves $\mathcal{C}_{t}$. Since we have that $\mathfrak{C}_{0,1}(1)=\mathfrak{C}_{0,1}(x)=0$ for any $t>0$, we choose $K$ to be the closed linear subspace generated by the first two modes.

Note that a basis for the space of polynomials in $x$ is the one generated by derivatives Chebyshev polynomials of the first kind denoted $T_k'(x)$. Hence each polynomial $f(x)$ of degree $N$ can be decomposed as $f(x)=\sum_{i=0}^{N}c_i T_{i+1}'(x)$. We denote $f[2]$ the truncated polynomial $f[2](x)=\sum_{i=2}^N c_i T_{i+1}'(x)$. In the rest of the proof we use the rescaled Chebyshev polynomials $T_k^{(t)}(x):=T_k\!(\tfrac{x}{2\sqrt t})$ for $t>0$, so that on $\mathcal C_t$ we have
\begin{equation*}
T_k^{(t)}(x(z))=\cos(k\theta),
\hspace{3mm} z=\sqrt t\,e^{i\theta}.
\end{equation*}
Consider the function $\textup{P}:\mathcal{H}\to K$ defined by $\textup{P}[\mathfrak{G}](f)=\mathfrak{G}\big(f-f[2]\big)$. It is clear from the definition that this is a projection and that $\textup{P}[\mathfrak{G}](1)=\mathfrak{G}(1)$ and $\textup{P}[\mathfrak{G}](x)=\mathfrak{G}(x)$. We are left to check that $\textup{P}$ is an orthogonal projection and that the covariance of $\mathfrak{G}-\textup{P}[\mathfrak{G}]$ coincides with the covariance of $\mathfrak{C}_{0,1}$. We will need to verify both of these claims. To verify that $\textup{P}$ is an orthogonal projection we want to check that $\textup{Cov}\big(\mathfrak{G}(f)-\textup{P}[\mathfrak{G}](f),\textup{P}[\mathfrak{G}](g)\big)=0$ for any $f$ and $g$ polynomials in $x$ supported in $\mathcal{C}_t$.

Equivalently, we need to check that for any $0\leq t_1,\,t_2\leq 1$, $k\geq 2$ and $k'=0,\,1$ we have
\begin{equation*}
\begin{split}
\oint\limits_{\substack{z\in \mathbb{H},\\\, |z|^2=t_1}} \oint\limits_{\substack{w\in \mathbb{H},\\\, |w|^2=t_2}} T_{k+1}'\big(x(z)\big) T_{k'+1}'\big(x(w)\big)\bigg(\frac{-1}{2\pi}\ln\Biggr| \frac{z-w}{z-\bar{w}}\Biggr|\bigg) \frac{dx(z)}{dz}\frac{dx(w)}{dw}\,dz\,dw=0.
\end{split}
\end{equation*}

We now reverse the computation used in the proof of Theorem \ref{TheoremModel}, we separate the integral depending on the sign of the imaginary part of $z$ and $w$ and integrate by parts to recover the identity
\begin{equation*}
    \frac{1}{(2\pi i)^2}\oint\limits_{\substack{ |z|^2=t_1}}\oint\limits_{\substack{ |w|^2=t_2}} T_{k+1}\big(z+\frac{t_1}{z}\big)T_{k'+1}\big(w+\frac{t_2}{w}\big)\bigg(\frac{\min(t_1,t_2)}{t_1(\min(t_1,t_2)z/t_1-w)^2}\bigg)\,dz\,dw=0,
\end{equation*}
which follows from the orthogonality of Chebyshev polynomials, see \cite[Proposition 3]{Bo}. The computation to verify that the covariances coincide is similarly done by using the orthogonality relations of Chebyshev polynomials.
\end{proof}

\begin{proof}[Proof of Theorem \ref{Theorem2}, item (3)]
It is enough to check that the covariance of $\mathfrak{C}_{a,b}$ and of $\mathfrak{C}'_{a,b}=\mathfrak{G}+\sqrt{a-1}\eta_1+\sqrt{b-2}\eta_2$ match. Since $\mathfrak{G}$, $\eta_1$ and $\eta_2$ are independent, it is a direct computation that
\begin{equation*}
\begin{split}
\textup{Cov}\big(\mathfrak{C}'_{a,b}(\sqrt{t}e^{i\theta}),\mathfrak{C}'_{a,b}(\sqrt{t'}e^{i\phi})\big)
=\frac{-1}{2\pi}\ln\Biggr|\frac{\sqrt{t}e^{i\theta}-\sqrt{t'}e^{i\phi}}
{\sqrt{t}e^{i\theta}-\sqrt{t'}e^{-i\phi}}\Biggr|
+\frac{(a-1)\min(t,t')}{\pi\sqrt{tt'}}\sin(\theta)\sin(\phi)
\\+\frac{(b-2)\min(t,t')^2}{2\pi tt'}\sin(2\theta)\sin(2\phi).
\end{split}
\end{equation*}

The change of variables $z=\sqrt{t}e^{i\theta}$ and $w=\sqrt{t'}e^{i\phi}$, jointly with the identities
\begin{equation*}
\Im\Big(\frac{1}{z}\Big)=-\frac{\sin(\theta)}{\sqrt{t}}
\qquad\textup{and}\qquad
\Im\Big(\frac{1}{z^2}\Big)=-\frac{\sin(2\theta)}{t},
\end{equation*}
and their analogues for $w$, allows us to recover the covariance kernel of $\mathfrak{C}_{a,b}$.
\end{proof}

We now combine the ideas for the proof of items (2) and (3) to conclude the proof of the theorem.

\begin{proof}[Proof of Theorem \ref{Theorem2}, items (4) \& (5)]
We fix
\begin{equation*}
    \mathfrak{C}'_{a,b}
    =
    \mathfrak{G}
    +\sqrt{\frac{a}{1-a}}\eta_1
    +\sqrt{\frac{b-1}{2-b}}\eta_2.
\end{equation*}
We will need to consider the enlarged Gaussian space generated by $\mathfrak{G}$, $\eta_1$ and $\eta_2$. We need to verify that conditioning $\mathfrak{G}$ so that $\mathfrak{C}'_{a,b}$ has first and second modes $0$ produces the covariance kernel of $\mathfrak{C}_{a,b}$. We recall the discussion in the introduction: if $X$ and $\eta$ are independent centered Gaussian random variables with the same variance and $p\geq0$, then $\textup{Var}\big(X\mid X+p\eta=0\big)=\frac{p^2}{1+p^2}\textup{Var}(X)$. 

Now denote by $K$ the closed linear subspace of the Gaussian Hilbert space generated by the first and second modes of $\mathfrak{C}'_{a,b}$. Since $\eta_1$ and $\eta_2$ are supported respectively on the first and second modes, these modes are given by
\begin{equation*}
    \mathfrak{G}_1+\sqrt{\frac{a}{1-a}}\eta_1
    \hspace{2mm}\textup{ and }\hspace{2mm}
    \mathfrak{G}_2+\sqrt{\frac{b-1}{2-b}}\eta_2,
\end{equation*}
where $\mathfrak{G}_1$ and $\mathfrak{G}_2$ correspond to the first and second modes of the GFF $\mathfrak{G}$. The conditioned field is therefore
$\mathfrak{G}-\textup{P}_K[\mathfrak{G}]$. Note that conditioning with respect to $K$ acts independently on these two modes and leaves all higher modes unchanged. Applying the Gaussian conditioning identity above, the resulting covariance of the first mode is multiplied by $a$ and the covariance of the second mode is multiplied by $b-1$. Hence we obtain
\begin{equation*}
\frac{-1}{2\pi}\ln\Biggr| \frac{z-w}{z-\bar{w}}\Biggr|
-\frac{(1-a)\min\big(|z|^2,|w|^2\big)}{\pi}\Im\Big(\frac{1}{z}\Big)\Im\Big(\frac{1}{w}\Big)-\frac{(2-b)\min\big(|z|^4,|w|^4\big)}{2\pi}\Im\Big(\frac{1}{z^2}\Big)\Im\Big(\frac{1}{w^2}\Big),
\end{equation*}
which coincides with the covariance kernel of $\mathfrak{C}_{a,b}$. The proof of item $(5)$ follows by applying the conditioning argument above to one of the first two modes and the independent Gaussian perturbation argument from item $(3)$ to the other.
\end{proof}


\begin{thebibliography}{XXX}




\bibitem[AGZ09]{AGZ} Anderson G., Guionnet A., Zeitouni O., {\it{An Introduction to Random Matrices}}, Cambridge University Press  (2011).


\bibitem[BN25]{BN25} Berggren T., Nicoletti M., {\it{Gaussian Free Field and Discrete Gaussians in Periodic Dimer Models}}, preprint (2025) \url{https://arxiv.org/pdf/2502.07241v2}.


\bibitem[Bo14]{Bo} Borodin A., {\it{CLT for spectra of submatrices of Wigner Random Matrices}}, Mosc. Math. J. {\bf{14}} (2014) 29--38.

\bibitem[BB14]{BB14} Borodin A., Bufetov A., {\it{Plancherel representations of}} $U(\infty)$ {\it{and correlated Gaussian free fields}}, Duke Math. J.{\bf{163}} (2014) 2109--2158.

\bibitem[BF14]{BF} Borodin A., Ferrari P., {\it{Anisotropic growth of random surfaces in 2+1 dimensions}}, Communications in Mathematical Physics, {\bf{325}} (2014) 603--684.

\bibitem[BG15]{BoG} Borodin A., Gorin V., {\it{General  $\beta$-Jacobi Corners Process and the Gaussian Free Field}}, Comm. Pure Appl. Math. {\bf{68}} (2015) 1774--1844.



\bibitem[BPZ25]{BPZ25} Bufetov A., Petrov L., Zografos P., {\it{Domino Tilings of the Aztec Diamond in Random Environment and Schur Generating Functions}}, preprint (2025) \url{https://arxiv.org/pdf/2507.08560}.


\bibitem[DKKMO20]{DKKMO20} Duminil--Copin H., Kozlowski K. K., Krachun D., Manolescu I., Oulamara M., {\it{Rotational invariance in critical planar lattice models}},  preprint (2020) \url{arxiv.org/abs/2012.11672}.

\bibitem[DKLM26]{DKLM26} Duminil--Copin H., Kozlowski K. K., Lammers P., Manolescu I., {\it{Gaussian free field convergence of the six-vertex model with }} $-1\leq \Delta\leq-\tfrac{1}{2}$,  preprint (2026) \url{arxiv.org/abs/2603.06268}.

\bibitem[DP18]{DP18} Dumitriu I., Paquette E., {\it{Spectra of overlapping Wishart matrices and the Gaussian free field}}, Random Matrices Theory Appl. {\bf{7}} 1850003 (2018).



\bibitem[Go21]{Gorin} Gorin V., {\it{Lectures on random lozenge tiling}}, Cambridge Univ. Press (2021).

\bibitem[GN25]{GN25} Gorin V., Nicoletti M., {\it{Six-vertex model and random matrix distributions}}, Bull. Amer. Math. Soc. {\bf{62}} (2025) 175--234.











\bibitem[Ke09]{Ke09} Kenyon R., {\it{Lectures on dimers}}, (2009) Available at  \url{arxiv.org/pdf/0910.3129}.



\bibitem[KZ23]{KZ23} Kuan J., Zhou Z., {\it{Three-dimensional Gaussian fluctuations of spectra of overlapping stochastic Wishart matrices}}, Random Matrices Theory Appl. {\bf{12}} 2250048 (2023).



\bibitem[LG16]{LG} Le Gall J-F., {\it{Brownian Motion, Martingales, and Stochastic Calculus}}, Graduate Texts in Mathematics, Springer International Publishing, {\bf{274}} (2016).

\bibitem[LRS20]{LRS20} Li L., Reed M., Soshnikov A., {\it{Central Limit Theorem for Linear Eigenvalue Statistics for Submatrices of Wigner Random Matrices}}, Front. Appl. Math. Stat. {\bf{6}} (2020) 17.












\bibitem[PW21]{PW} Powell E., Werner W., {\it{Lecture notes on the Gaussian free field}}, Cours spécialisés, Société Mathématique de France {\bf{28}} (2021).

\bibitem[Ra25]{RaMain} Raposo G., {\it{Global fluctuations for standard Young tableaux}}, preprint (2025) \url{arxiv.org/pdf/2507.18601}.

\bibitem[Re12]{Re12} Regev A., {\it{A proof of Catalan’s Convolution formula}}, Integers {\bf{12}} (2012) 929--934.




\bibitem[Sh07]{Sh07} Sheffield S., {\it{Gaussian free fields for mathematicians}}, Probab. Theory Related Fields,
{\bf{139}} (2007) 521--541.







\end{thebibliography}
\end{document}